\documentclass[12pt]{amsart} 

\usepackage{amsmath,amstext,amscd, amssymb,txfonts, epsfig, psfrag, color, epstopdf}  

\usepackage[left=3cm, top=2cm, right=3cm, bottom=3cm]{geometry}

  \usepackage{tikz}
\usetikzlibrary{decorations.markings}
\usetikzlibrary{decorations.pathreplacing}
\usetikzlibrary{arrows.meta}
\usetikzlibrary{calc}
\usetikzlibrary{patterns}
  \usepackage{tkz-euclide} 

\usepackage{thmtools, thm-restate}

\theoremstyle{plain}

\newtheorem{thm}{Theorem}[section]
\newtheorem{lemma}[thm]{Lemma}

\newtheorem{cor}[thm]{Corollary}
\newtheorem{prop}[thm]{Proposition}

\newtheorem{scholium}[thm]{Scholium}

\newtheorem{theorem}[thm]{Theorem} 
\newtheorem*{theorem*}{Theorem}

\numberwithin{equation}{section}

\newtheorem{letterthm}{Theorem}%[section]

\theoremstyle{definition}
\newtheorem{defn}[thm]{Definition}
\newtheorem{definition}[thm]{Definition}

\newtheorem{remark}[thm]{Remark}

\newtheorem{example}[thm]{Example}

\newtheorem{Open questions}[thm]{Open questions}
\newtheorem{Open question}[thm]{Open question}
\newtheorem{Open problems}[thm]{Open problems}
\newtheorem{Open problem}[thm]{Open problem}

\def\dc{{\mathfrak{m}}}   % this used to be the departure function {{\rm{DCyc}}}

\def\g{{\gamma}}
\def\Z{\mathbb{Z}}
\def\N{\mathbb{N}}

\def\Area{\hbox{\rm Area}}

\def\CL{\hbox{\rm CL}} 
 
\def\F+L{\hbox{$\textup{F}\!_+\textup{L}$}}

\def\ssm{\smallsetminus}

\def\e{\varepsilon}

\def\l{\lambda}

\def\onto{{\kern3pt\to\kern-8pt\to\kern3pt}}

\def\<{\langle}
\def\>{\rangle}
\def\|{{\ |\ }}

\def\half{{\frac{1}{2}}}

\def\a{\alpha}

\def\g{\gamma}
\def\G{\Gamma}
\def\d{\delta}

\def\m{{\mathfrak{m}}}

\def\F3{${\rm{F}}_3$ }

\def\disp{{\rm{Dist}}_P^{G\times G}}
\def\dist{{\rm{Dist}}}

\def\*{^{\star}}

\begin{document}

\title{Conjugacy in fibre products, distortion, 
and the geometry of cyclic subgroups}

% author  information
\author[Martin R.   Bridson]{Martin R.  ~Bridson}
\address{Mathematical Institute\\
Andrew Wiles Building\\
Woodstock Road\\
Oxford,   OX2 6GG}
\email{bridson@maths.ox.ac.uk}

\keywords{fibre products,  conjugator length,  cyclic subgroups,  rel-cyclics Dehn function}

\subjclass[2010]{20F67,   20F10,  20F65} 

%\date{Edited following referee comments,  May 2026}  

\begin{abstract} 
We investigate the complexity of the conjugacy problem for 
fibre products in torsion-free hyperbolic groups. 
Let $G$ be a torsion-free hyperbolic group and let $P<G\times G$ be the fibre product
associated to an epimorphism $G\onto Q$.  
We establish  inequalities  that relate the conjugator length function of $P$ to the geometry of  cyclic subgroups in $Q$,  the Dehn function of $Q$,  the {\em rel-cyclics Dehn function} of $Q$, and the distortion of $P$ in $G\times G$.  These estimates provide tools for
extending the library of (large) functions that are known to arise as the conjugator length functions  of finitely generated and finitely presented groups.
\end{abstract}

\maketitle

\section{Introduction}  If a subdirect product of torsion-free hyperbolic groups
$P<G_1\times\dots\times G_m$  
projects to a  subgroup
of finite index in each pair of factors $G_i\times G_j$, then $P$ is finitely
presented and its conjugacy problem can be solved using algorithms in an 
associated virtually nilpotent group \cite{BM, BHMS}.
If $P$ does not virtually surject to each $G_i\times G_j$, then the situation is
much wilder. It is this wildness that we seek to quantify in this article,
concentrating on the case  $m=2$.

In the  case $m=2$, there
is a  correspondence between subdirect products and {\em fibre products}: if $P<G_1\times G_2$ is subdirect
then there is an isomorphism
$G_1/(P\cap G_1)\cong G_2/(P\cap G_2)$ and $P$ is the
fibre product of  the maps $p_i:G_i\onto G_i/(P\cap G_i)$;
see \cite{BM}.
The quotient $Q:=G_1/(P\cap G_1)$ will be finitely presented if and only if $P$ is finitely generated. 

It has been known for a long time that if the word problem in $Q$ in unsolvable, then the conjugacy problem in $P$
is unsolvable \cite{mihailova, cfm:thesis, BBMS},  but the following characterisation of exactly when the conjugacy
problem in $P$ is solvable was only proved  recently \cite{mb:iff}.

\begin{letterthm}[\cite{mb:iff}] \label{t:iff}  Let  $G_1$ and $G_2$
be torsion-free hyperbolic groups,    let $P<G_1\times G_2$ be a finitely generated
subdirect product, and let $Q$ be the finitely presented group $G_1/(G_1\cap P)$.   Then,   
the following conditions are equivalent:
\begin{enumerate}
\item the conjugacy problem in $P$ is solvable;
\item there is a uniform algorithm to decide membership of cyclic subgroups in $Q$;
\item  the rel-cyclics Dehn function $\delta_Q^c(n)$ is recursive.
\end{enumerate}  
\end{letterthm}

The main purpose of this article is to provide quantitative companions to this theorem, concentrating
on the case where $G_1=G_2$ and $p_1=p_2$. These results are intended as tools 
that can be used to  extend the library of functions that are
known to arise as conjugator length functions of finitely generated and finitely presented groups \cite{BRS},
a project that is being pursued with Tim Riley \cite{BR}. 
By definition, the {\em conjugator length function} of a group $G$ is
$$
\CL_G(n) := \max_{u,v} \min_\g \{ |\g|_G \colon u,v,\g\in G,\ \g^{-1}u\g=v,\ 
|u|_G+|v|_G\le n\},
$$ 
where $|\g|_G$ denotes distance from the identity in the
word metric associated to a fixed finite generating set for $G$.  
Given a pair of finitely generated groups $G<\G$, 
we also consider the function $\CL_G^\G(n)$ obtained by replacing 
the condition $|u|_G+|v|_G\le n$ with $|u|_\G+|v|_\G\le n$, still measuring $|\gamma|_G$ in $G$.
We refer to \cite{BRS} for an account of what is known about conjugator length functions.

An important theme running throughout this article is that the complexity of the
conjugacy problem in a fibre product $P$ depends not only on the nature of the word problem in 
the quotient $Q$,
but also on the geometry of cyclic subgroups  in $Q$.  This is why our 
analysis involves not only the
Dehn function $\delta_Q(n)$ of $Q$ but also the {\em rel-cyclics Dehn function} $\d_Q^c(n)$.
Whereas 
$\delta_Q(n)$ bounds the number of relators of $Q$ that one
must apply to reduce null-homotopic  words $w$ with $|w|\le n$ to the empty word,
$\d_Q^c(n)$ bounds the number of relators that one must apply to reduce null-homotopic  words
of the form $wu^{-p}$ to the empty word, where $|w| + |u|\le n$ and $|p|$ is minimal;
see Definition \ref{d:d^c}.    

By design,  
$\d_Q^c(n)$ controls both the Dehn function $\d_Q(n)$ and the distortion of cyclic subgroups in $Q$.
We will be particularly interested in groups that have 
{\em uniformly quasigeodesic cyclics}, i.e.~there is a constant $\lambda>0$
such that $|q^n|_Q\ge \lambda n$ for all $q\in Q\ssm\{1\}$. Torsion-free
hyperbolic and CAT$(0)$ groups have this  property, as do their finitely
generated subgroups and many other groups (see Section \ref{s:UQC}).
In the  
hyperbolic case, infinite cyclic subgroups enjoy an additional monotonicity property
that will be crucial for us (Proposition \ref{p:roots}): there is a constant $k>0$ such that $|h^i|_H\le k|h^p|_H$ for all $h\in H$ and all $0<i<p$.

It is standard practice in geometric group theory to write $f \preceq g$ for functions $f,g:\N\to\N$
if there exists a 
positive integer  $C$ such that
$
f(n) \leq C g(Cn+ C) + Cn + C
$
for all $n \in \N$. One writes $f\simeq g$ if  $f \preceq g$ and $g  \preceq f$.

\begin{restatable}{letterthm}{mainthm} \label{t:main}
Let $Q$ be a finitely presented group, let $G\onto Q$ be an epimorphism from a 
torsion-free hyperbolic group, and let $P<G\times G$ be the associated fibre
product. Then,
$$
\d_Q(n) \preceq \CL_P(n)\preceq  \CL_P^{G\times G}(n) \preceq \d_Q^{c}(n).
$$
If  $Q$ is torsion-free, then 
$$ 
\d_Q^{c}(n )\preceq n\, \CL_P^{G\times G}(n), 
$$ 
and if $Q$ has uniformly quasigeodesic cyclics, then 
$$
\delta_Q^{c}(n)\simeq \CL_P^{G\times G}(n),
$$
provided $\CL_P^{G\times G}(n)\succeq n^2$.
\end{restatable}

\begin{restatable}{lettercor}{maincor}\label{c:forBR} 
If $Q$ has uniformly quasigeodesic cyclics, then 
$$
\d_Q(n) \preceq \CL_P(n) \preceq    \d_Q(n^2).
$$  
\end{restatable}

Theorem \ref{t:main}  is  proved in Section \ref{s:summary} by bringing together estimates established
in Sections \ref{s:3} to \ref{s:upper}.  Some of these estimates contain finer information.

Finitely presented groups are the main objects of interest in geometric and
combinatorial group theory and, correspondingly, 
the preceding results are of most interest when the fibre product
$P$ is finitely presented.  
{\em{The 1-2-3 Theorem}} from \cite{BBMS} states that $P$ will be finitely presented if three conditions hold:
$G$ is finitely presented,  $Q$ is of type ${\rm{F}}_3$ (i.e.~there
is a classifying space $K(Q,1)$ with a finite 3-skeleton), and the  kernel of $G\onto Q$ is finitely generated.  We shall
appeal to this criterion repeatedly in this article.

For the moment, it is unclear whether one can arrange for the Dehn function of $Q$ to be
as unconstrained as one would like if one has to ensure both
that $Q$ is of type ${\rm{F}}_3$ and that its cyclic subgroups
are controlled in the manner needed for our results. To circumvent this
problem, we describe a construction that establishes the following.

\begin{restatable}{letterprop}{makeUQC}\label{p:makeUQC}
There is an algorithm that, given a finite presentation for a
group  $Q$, will produce a finite presentation for a group $Q^\dagger$, with 
$Q^\dagger\onto Q$, so
that $Q^\dagger$ has uniformly quasigeodesic cyclics and
$$n\ \d_Q(n) \preceq \d_{Q^\dagger} (n) \preceq n\ (\d_Q(n))^2.$$
If $Q$ is of type ${\rm{F}}_3$, then so is $Q^\dagger$.
\end{restatable}

Theorem \ref{t:main} and  Proposition \ref{p:makeUQC}   
rely on the following calculation of {\em distortion} for fibre products,  variations of which
appear elsewhere in the literature. By definition,
$\disp (n) := \max\{|\gamma|_P \colon |\gamma|_{G\times G}\le n\}$.

\begin{restatable}{letterthm}{distort}\label{t:Pdist}  
Let $Q$ be a finitely presented group. If $G$ is hyperbolic and $P<G\times G$ is the fibre product of an epimorphism
$G\twoheadrightarrow Q$, then 
$\disp (n) \simeq \d_Q(n).$
\end{restatable}

In Section \ref{s:last} we combine the dagger construction of Proposition \ref{p:makeUQC}  with 
Corollary \ref{c:forBR} and a template for 
designer groups described in \cite{mrb:icm}.
This enables us to provide the following tool for investigations into the range of behaviour 
that conjugator length functions of finitely presented groups can exhibit. 
A Dehn function $\d(n)$ is {\em standard}
if $\d(n)\simeq\d_Q(n)$ for some group $Q$ of type ${\rm{F}}_3$. All known
Dehn functions are standard; see  Remark \ref{r:standard}.

\begin{restatable}{letterthm}{final}\label{t:final}
For every standard Dehn function $\d(n)$, there exists a finitely presented
group $P$ with
$$
n\ \d(n) \preceq \CL_P(n) \preceq n^2 (\d(n^2))^2.
$$
\end{restatable}

The group $P$ in Theorem \ref{t:final} will be a 
subgroup of $G\times G$ for some torsion-free hyperbolic group $G$.  Furthermore, $G\times G$ will be
virtually special in the
sense of Haglund and Wise \cite{HW}.  It follows that $P$ is residually finite
and can be embedded
${\rm{SL}}(d,\Z)$ and in the mapping class group of
a closed surface of some genus \cite{mrb:mrl}, where $d$ and the
genus depend on  $P$.

If $\d(n)\succeq 2^n$,  then,  up to $\simeq$ equivalence,  the difference between
the upper and lower bounds in Theorem \ref{t:final}
lies entirely with the fact that the argument of 
$\d$ in the upper bound is $n^2$ rather than $n$. This gap originates in
Corollary \ref{c:forBR}, so to close the gap one would have to identify reasonable
hypotheses on $Q$ that allow the  upper bound in that result to be improved.  This might lead to a proof that the
conjugator length functions of finitely presented groups 
are as diverse as Dehn functions,  a fact that has been established by different means in
recent work of Gillis and Wagner \cite{GW}.

\noindent{{\bf{Acknowledgements.}} 
I thank Chuck Miller for the many insights that I gained
from him during our long collaboration on fibre products and decision problems, and I thank Tim Riley for the many ideas
that he has shared during our collaboration on conjugator length functions.  It has been a joy to work with them both.
I am grateful to the University of Oxford for the sabbatical leave that enabled me to finish this work,
to Stanford University for hosting me during this sabbatical, and to the Lieb family for their boundless hospitality
during my time at Stanford.  
I also thank the referees of this article for their careful reading and helpful comments.

\section{The rel-cyclics Dehn function}\label{s:2}  
 
We are interested in fibre products $P<G\times G$ of maps $G\twoheadrightarrow Q$,
particularly features of $P$ that are related to aspects of the word problem in $Q$. 
In order to quantify these connections, we need to recall some  
standard facts and terminology concerning the geometry of the word problem
in finitely presented groups. The reader unfamiliar with this material may wish to consult
\cite{mrb:bfs} or \cite{BRSh}.

Let $\mathcal{P} \equiv \< X \mid R\>$ be a finite presentation of a group $\G$. By definition,
a word $w$ in the free group $F(X)$ represents the identity in $\G$ if and only if 
there is
an equality in $F(X)$ 
\begin{equation}\label{w:equ}
w = \prod_{i=1}^M \theta_i^{-1} r_i \theta_i
\end{equation}
with $r_i\in R\cup R^{-1}$ and $\theta_i\in F(X)$.  
The number $M$ of factors in the product on the righthand side is defined to be 
the {\em{area}} of the product. 
This terminology is motivated by the   correspondence with diagrams
that comes from van Kampen's Lemma \cite{BRS, mrb:bfs}.

One defines
${\rm{Area}}(w)$ to be the least area among all products of this form for $w$.
The {\em{Dehn function}} of $\mathcal{P}$ is the function $\delta: \N \to \N$ defined by
$$
\delta(n) = \max \{{\rm{Area}}(w) \mid w=_\G1 \ {\rm{and}} \ |w|\le n\}.
$$

The {\em noise} of the product  on the right of (\ref{w:equ}) is defined to be 
$\sum_{i=0}^{M} |\theta_{i}\theta_{i+1}^{-1}|$, with the convention that 
$\theta_0$ and $\theta_{M+1}$ are the empty word. 
An analysis of the standard proof of van Kampen's Lemma  
yields the {\em Bounded Noise Lemma}, an observation (and terminology) that is
due to the authors of \cite{BBMS}; see \cite{dison}, p.18
for an explicit proof.

\begin{lemma}[Bounded Noise Lemma]\label{l:BNL}
If ${\rm{Area}}(w) = M$, then there is a product of the form (\ref{w:equ})
with area $M$ and noise at most $ML + |w|$, where $L$ is the length of the longest relator in $R$.
\end{lemma}

\subsection{The rel-cyclics Dehn function}

The following functions will arise naturally in our study of 
fibre products.

\begin{definition}\label{d:d^c}
The {\em rel-cyclics} Dehn function of a finitely presented group $\G=\<X\mid R\>$ is 
$$ 
\d^{c}(n) : = \max_{w,u} \{ {\rm{Area}}(w\,u^{-p}) +|pn| \colon |w|+|u|\le n,\ w=_\G u^{p},\
|p|\le  o(u)/2\},
$$
where $o(u)\in\N\cup\{\infty\}$ is the order of $u$ in $\G$.  (The condition $|p|\le  o(u)/2$
is equivalent to requiring that $|p|$ is the least non-negative  integer such that  $w=u^{\pm p}$ in $\G$.) 
The {\em rel $\infty$-cyclics} Dehn function  will also enter our discussion,
$$ 
\d^z(n) : = \max_{w,u} \{ {\rm{Area}}(w\,u^{-p}) +|pn| \colon |w|+|u|\le n,\ w=_\G u^{p},\
o(u)=\infty\},
$$
as will the following variant of $\delta^c(n)$,
$$ 
\d^{o}(n) : = \max_{w,u} \{ {\rm{Area}}(w\,u^{-p}) +|pn| \colon |w|+|u|\le n,\ w=_\G u^{p},\
|p|\le  o(u)\}.
$$
Note the dependence of these functions on the  {\em return of cyclics} function of $\G$, 
which is defined as
$$
\m(n) := \max \{ p \colon \exists u\in\G,\ o(u)=\infty,\ |u|_\G\le n,\, |u^p|_\G\le n\}.
$$
$\d^{o}(n)$ is  also related to the {\em torsion-evolution} function of $\G$, which is
$$
{\mathfrak{t}}(n) := \max \{ o(u) \colon |u|_\G\le n,\, o(u)<\infty\}.
$$
\end{definition}

The Dehn function is a group invariant in the sense that the Dehn functions associated to different finite
presentations of a fixed group are $\simeq$ equivalent; 
see e.g.~\cite{mrb:bfs, BRSh}. A similar argument shows that $\d_\G^{c}(n)$ and its variants
are group invariants in the same sense.
This justifies the notation 
$\d_\G(n)$ and $\d_\G^{c}(n)$ in the statements of our theorems,
where only the $\simeq$ equivalence class is of concern. Likewise, it is
easy to verify that, up to $\simeq$ equivalence, $\m(n)$ and  ${\mathfrak{t}}(n)$ are independent of
the chosen finite generating set.  The $\simeq$ equivalence of $\d_\G^{c}(n)$ is also indifferent
to whether one quantifies over words $w,u$ with  $\max\{|w|,\, |u|\}\le n$ rather than $|w|+|u|\le n$.
 
We will often use a subscript on each of the  functions discussed above
to emphasize which group the function is attached to,  suppressing mention of
the finite presentation chosen.

\begin{remark}[{\em{On the definitions of $\delta^{c}(n)$ and $\delta^{o}(n)$}}]\label{r:2.3}
In the definition of $\delta^{o}(n)$,  by taking $w$ to be the empty word we see that
$\delta^{o}(n)\ge {\rm{Area}}(u^{o(u)}) + {\mathfrak{t}}(n)$ for each torsion element $u$ with $|u|\le n$. 
The more stringent condition $|p|\le o(u)/2$ in  $\delta^{c}(n)$
decouples $\delta^{c}(n)$ from ${\rm{Area}}(u^{o(u)})$ and removes the direct comparison with ${\mathfrak{t}}(n)$,
treating finite and infinite cyclic subgroups in a more equitable manner.   The fact that $\delta^{c}(n)$
does not control ${\mathfrak{t}}(n)$ is an essential feature of Theorem \ref{t:iff}, because there exist fibre products
$P<G_1\times G_2$ that satisfy the conditions of that theorem but are such that 
there is no algorithm to decide the orders of elements in $Q=G_1/(G_1\cap P)$: this
means that  $\delta_Q^{c}(n)$ is recursive  (i.e.~computable) but  ${\mathfrak{t}}_Q(n)$ is not -- see \cite{collins}
and \cite[Corollary B]{mb:iff}. 

\end{remark} 
 
 \begin{remark}[{\em{On the geometric content of $\delta^{c}(n)$}}] The Dehn function of a finitely
 presented group $\G$ constrains the geometry of 
 discs in any closed Riemannian manifold $M$ with fundamental group $\G$: up to $\simeq$ equivalence,  $\delta_\G(n)$
 gives the optimal upper bound on the area of minimal-area discs filling loops of length at most $n$ in the universal
 cover of $M$; this is Gromov's Filling Theorem -- see \cite{mrb:bfs} for a detailed proof. 
 
 The rel-cyclics Dehn function $\delta_\G^{c}(n)$ admits a similar interpretation,  but the set of loops that we are now filling
has been enlarged to account for the distortion of cyclic subgroups in $\G$:  whereas $\delta_\G(n)$ provides a measure of how difficult it is to shrink null-homotopic loops of length at most $n$ to a point in $M$,  the rel-cyclics Dehn function $\delta_\G^{c}(n)$ can be viewed as an estimate of how difficult it is to take a loop $w$ 
that is equal in $\pi_1M$ to a power, $p$ say,  of a  loop $u$ and homotop $w$ to the path that traces $p$ times around 
$u$; we quantify in terms of the sum of the lengths of $w$ and $u$.  The measurement that we use to quantify difficultly
  is the area of the homotopy,  which is viewed as a map of a disc to the universal covering $\widetilde{M}$; see figure \ref{fig1}. If $\<u\>$ is highly distorted in $\G=\pi_1M$,  then the length of $u^p$ can be much larger than the length of $w$ and $u$, 
so $\delta_\G^{c}(n)$ depends not only on the geometry of discs in $\widetilde{M}$ but also on the geometry of 
cyclic subgroups in $\G=\pi_1M$.  The additional summand $|pn|$ in the definition of $\delta_\G^{c}(n)$ is needed
for robustness: it accounts for the additional area that the disc might acquire near the boundary when $u$ is perturbed.
\end{remark} 

\begin{figure}[ht]
        \begin{center}
         \includegraphics[width=7.0cm]{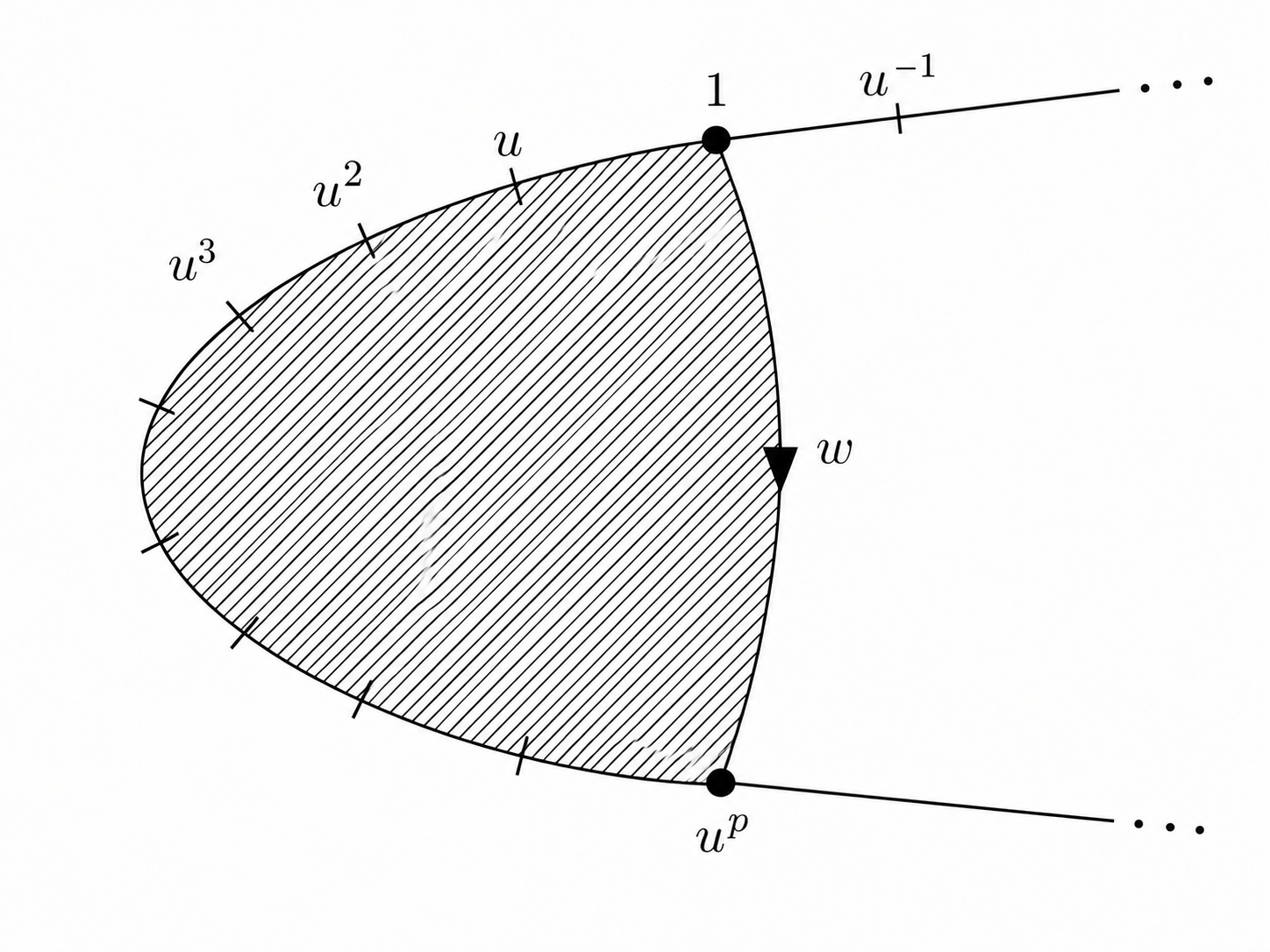}  
                      \end{center}
        \caption{The geometric interpretation of $\d^{c}(n) $}\label{fig1}
\end{figure} 

\begin{remark}[{\em{Standard Dehn functions}}]\label{r:standard}
We say that a Dehn function is {\em standard} if it is the Dehn function of a
group of type ${\rm{F}}_3$. It is unknown whether all Dehn functions
are standard, up to $\simeq$ equivalence, but all of those in the literature 
are.

In \cite{BRSa}, Birget, Rips and Sapir went a long way towards
characterising the functions
which are Dehn functions of finitely presented groups. Their
criterion for being a Dehn function
is expressed in terms of 
time functions of Turing machine, with a super-additivity condition
and a technicality 
related to the ${\bf{P}}$ versus ${\bf{NP}}$ problem that only affects sub-exponential
Dehn functions.
This criterion covers all of the known Dehn functions 
$\d(n)\succeq n^4$.
In addition, the groups constructed in \cite{BRS} are all fundamental
groups of finite aspherical complexes, as are the groups in \cite{BB}, so 
their Dehn functions $\d(n)$ are standard.
\end{remark}

\section{Distortion of Fibre Products}\label{s:3}

We fix an exact sequence
 $1\to N\to G\overset{f}\to Q\to 1$ 
and consider the fibre product $P<G\times G$. Although we shall not mention it explicitly,
on several occasions we will exploit the observation that
$P = (N\times 1)\rtimes \Delta$, where $\Delta\cong G$ is the diagonal in $G\times G$.
This semidirect product decomposition is obtained by splitting the projection of $P$ onto $1\times G < G\times G$.

If $G$ and $Q$
are finitely presented, then there is  
a finite subset $A\subset N$ that generates $N$ as a normal subgroup.
We choose a finite presentation $G=\< X\mid R\>$ with
$A\subset X$, and we
present $Q$ as $Q=\< X\mid R\cup A\>$, so 
the identity map on $X$ induces $f:G\to Q$. 
The following well-known observation is easily verified. %(cf.~\cite{fritz}).

\begin{lemma}\label{l:3.1}
$P$ is generated by $\{(a,1), (x,x)\mid a\in A,\, x\in X\}$.
\end{lemma}
  
\def\dis{{\rm{Dist}}_P^{G\times G}} 

The  bound $\dis (n) \preceq \d_Q(n)$ in the following theorem remains valid
when $G$ is not hyperbolic,  and Remark \ref{r:upp} explains what one can say about the reverse inequality. 

\distort*
 
\begin{proof} 
Working with the generators $\{(1,x), (x,1)\mid
x\in X\}$ for $G\times G$ and
$\{(a,1), (x,x)\mid a\in A,\, x\in X\}$ for $P$,
we have $|(\gamma_1,\gamma_2)|_{G\times G}
= |\gamma_1|_G + |\gamma_2|_G$ and  $|(\gamma, \gamma)|_P= |\gamma|_G$. 
To see this last equality, note that if one deletes the generators $(a,1)$ from a word $w$ 
representing $(\gamma,\gamma)\in P$, then the projection to $1\times G$ still represents $\gamma$,
and therefore the redacted word still represents $(\gamma,\gamma)$. Thus only the diagonal 
generators $(x,x)$ occur in geodesic words for $(\gamma,\gamma)$.

Suppose $(\gamma_1,\gamma_2)\in P$ and let $\gamma = \gamma_2^{-1}\gamma_1 \in N$.
From the triangle inequality in $G$, 
% \begin{align*}\label{e:GxG} 
 %|(\gamma,1)|_{G\times G} =&
$$|\gamma|_G\le |\gamma_1|_G + |\gamma_2|_G = |(\gamma_1, \gamma_2)|_{G\times G}.$$
% =&  \le |(\gamma_1,\gamma_2)|_{G\times G} + |(\gamma_2,\gamma_2)|_{G\times G}\\
%=& |\gamma_1|_G + 3\, |\gamma_2|_G \le 3\,|(\gamma_1, \gamma_2)|_{G\times G}.
%\end{align*}	 
And from the triangle inequality  in $P$,  noting that $(\gamma,1)= (\gamma_2,\gamma_2)^{-1}(\gamma_1,\gamma_2)$,  
$$
\big{|}|(\gamma, 1)|_P - 
|(\gamma_1,\gamma_2)|_P\big{|}  \le 
|(\gamma_2,\gamma_2)|_P  = |\gamma_2|_G \le  |\gamma_1|_G + |\gamma_2|_G =  |(\gamma_1, \gamma_2)|_{G\times G}.
$$
Thus, if $|(\gamma_1, \gamma_2)|_{G\times G}\le n$, then $|\gamma|_G\le n$ and the difference between $|(\gamma,1)|_P$
and $|(\gamma_1,\gamma_2)|_P$ is at most $n$. Therefore it is enough to establish upper and lower bounds on
$|(\gamma, 1)|_P$ for  $\gamma\in N$ with $|\g|_G\le n$.

Let $w$ be a word of length at most $n$ in the generators $X^{\pm 1}$ with $w=\gamma\in N$.
As  $w=1$ in $Q=\< X\mid R\cup A\>$, there is an equality
in the free group $F(X)$,
\begin{equation}\label{e2}
w = \prod_{i=1}^M \theta_i^{-1} r_i \theta_i,
\end{equation}
with $r_i\in (R\cup A)^{-1}$ and $M\le \d_Q(n)$. By Lemma \ref{l:BNL},
we may assume that the noise 
$\sum_i |\theta_i \theta_{i+1}^{-1}|$ 
is at most $n+\d_Q(n)L$, where $L$ is the length of the longest relator in $R$. 
If $r_i\in R^{\pm 1}$, then $r_i=1$ in $G$. So 
by deleting  each factor with $r_i\in R^{\pm 1}$ from the product 
and reindexing we get an
equality in $G$ of  the form
\begin{equation}\label{e22}
w = \prod_{j=1}^{M'} \theta_j^{-1} a_j \theta_j,
\end{equation}
with $a_j\in A^{\pm 1}$ and $M'\le M$. The noise of the product in (\ref{e22}), which we still regard as an element 
of the free group,  is no greater
than that of (\ref{e2}) because if the
factor $\theta_i^{-1} r_i \theta_i$ is deleted, then  $|\theta_{i-1}
\theta_{i}^{-1}| + |\theta_i
\theta_{i+1}^{-1}|$ is subtracted from the noise while
$|\theta_{i-1}\theta_{i+1}^{-1}|$ is added, and by the triangle
inequality 
$$
|\theta_{i-1}
\theta_{i}^{-1}| + |\theta_i
\theta_{i+1}^{-1}|\ge |\theta_{i-1}\theta_{i+1}^{-1}|.
$$ 

We can regard (\ref{e22}) as an equality
in $G\times 1$ by replacing each symbol $x$ with $(x,1)$.
And since $(x,x)^{-1} (a,1) (x,x) = (x,1)^{-1}(a,1) (x,1)$ for all $x\in X$ and
$a\in A$, the equality remains valid in $G\times G$ if we subsequently replace
each symbol $(x,1)$ in each of the $\theta_j$ by $(x,x)$ while replacing $a_j$ with $(a_j,1)$.
Writing $\Theta_j$ for the word derived from $\theta_j$ in this way, we
obtain an equality in $P<G\times G$ where the righthand side is a product of conjugates of generators:
$$
(\gamma,1) = (w,1) = \prod_{j=1}^{M'} \Theta_j^{-1} (a_j,1) \Theta_j,
$$
where $M'\le M \le \d_Q(n)$. Moreover, since each $|\Theta_j\Theta_{j+1}^{-1}| = |\theta_j
\theta_{j+1}^{-1}|$, the noise $\sum_j |\Theta_j\Theta_{j+1}^{-1}|$
is no greater than the noise of (\ref{e2}), which is at most $L\d_Q(n) +n$.
Thus, recalling the convention that $\theta_0 $ and $\theta_{M'+1}$
are empty words, we get 
\begin{equation}\label{e-new}
|(\gamma, 1)|_P \le \sum_{j=0}^{M'} (|\theta_j\theta_{j+1}^{-1}| + 1) \le (L+1)\d_Q(n) +n.
\end{equation} 
This finishes the proof that $\dis (n)\preceq \d_Q(n)$.
\def\om{\omega}
\def\omn{\omega_n} 

Towards proving that $\d_Q(n)\preceq \dis (n)$, 
consider a word $\omn\in F(X)$ of length at most $n$ 
such that $\omn=1$ in $Q$ and
$\Area(\omn) = \d_Q(n)$. Note that $\omn\in N<G$.
Let $W$ be a word of length $\ell = |(\omn, 1)|_P$
in the generators $A_0=\{(a,1)\mid a\in A\}$ and $X_\Delta = \{
(x,x)\mid x\in X\}$ that equals $(\omn, 1)$ in $P$. 
We move all occurrences of each $(x,x)$ to the left in $W$ to obtain
an equality {\em in the free group} $F(A_0\cup X_\Delta)$ 
\begin{equation}\label{e3}
W =  W_0 \prod_{i=1}^{\ell'} u_i^{-1} \a_i u_i
\end{equation}
where $W_0$ is a word of length at most $\ell$ in the symbols $(x,x)$,
each $\a_i\in A_0^{\pm 1}$ occurred in the unedited word $W$, and $\ell'\le \ell$.
Now $W=(\omn,1)$ in $G\times G$, so by reading the second coordinate we
see that $W_0=1$ in the diagonal subgroup of $G\times G$. As $G$ is hyperbolic, 
there is a {\em free} equality
$$
W_0 = \prod_{j=1}^{\ell''} v_j^{-1} r_j^\Delta v_j
$$
with $r_j^\Delta\in R_\Delta^{\pm 1}\subset F(X_\Delta)$
and $\ell''\le C\ell'$, where $C$ the isoperimetric constant for $G$.
 By substituting this expression into (\ref{e3}) and reading
the first coordinate of each symbol, we obtain an equality in the free group $F(X)$
\begin{equation}\label{eq4}
w_1 = \prod_{j=1}^{\ell''} (v_j^{-1} r_j v_j) \ \prod_{i=1}^{\ell'} (u_i^{-1} a_i u_i)
\end{equation}
with $r_i\in R^{\pm 1}$ and $a_i\in A$, where $w_1$ is the word
of length $\ell$ that one obtains
by reading the first coordinate in $W$,  transcribing 
$(a,1)$  to $a$
and $(x,x)$ to $x$. 

By definition, $\omn=w_1$ in $G$, so using the linear isoperimetric inequality of
$G$ again, we have an equality in $F(X)$:
$$
\omn = w_1 \prod_{k=1}^B \mu_k^{-1} \rho_k \mu_k,
$$
with $B< C(n + \ell)$ and $\rho_k\in R^{\pm 1}$. By combining this with 
(\ref{eq4}), we get an equality expressing $\omn$ in the free group $F(X)$
as a product of $B + \ell'' + \ell' $ conjugates of
elements of $(R\cup A)^{\pm 1}$.  Recalling that $\ell'\le \ell$ and $\ell''\le C\ell'$, we deduce
 $\d_Q(n) = \Area(\omn) \le Cn+ \ell (2C + 1)$,
whence 
\begin{equation}\label{e:omn}
|(\omn, 1)|_P = \ell \ge \frac{1}{2C+1}(\d_Q(n) -Cn),
\end{equation}  
as required.
\end{proof}

\begin{remark}\label{r:upp}
 In the preceding proof, our only use of the hypothesis that $G$ was hyperbolic was our invocation of the linear isoperimetric inequality, first to get $\ell''\le C\ell'$ and then to produce the constant $B$.  For arbitrary
finitely presented groups $G$, one can instead  estimate using $\d_G(n+\ell)$,
leading to an inequality of the form $\d_Q(n)\preceq \d_G\circ\dis(n)$. 
\end{remark} 

Near  the beginning of the preceding proof, we made the following observation, which will be needed again later.

\begin{scholium}\label{l:Ndoes}
There is a sequence of elements $\gamma_n\in N$ with $|\gamma_n|_G\le n$
such that $|(\gamma_n, 1)|_P\ge \dis (n) - n$.
\end{scholium}

We shall also need the following summary of what we actually proved in the argument 
establishing the bound  $\dis (n)\preceq \d_Q(n)$.

\begin{scholium}\label{scholiumNew} If $(\gamma_1,\gamma_2)\in P$ and $|(\gamma_1,\gamma_2)|_{G\times G}\le n$,
then $|\g_2^{-1}\g_1|_G\le n$ and for any word  $w\in F(X)$ representing $\gamma_2^{-1}\gamma_1$ in $G$,  
$$
|(\gamma_1,\gamma_2)|_P \le |(\gamma_2^{-1}\gamma_1, 1)|_P + n \le (L+1) \Area_Q(w) + |w| + n.
$$
\end{scholium} 

\section{The geometry of cyclic subgroups} 

In this section we explore three geometric properties of cyclic subgroups of
finitely generated groups. The first property, having {\em uniformly quasigeodesic
cyclics}, is enjoyed by a wide range of groups, while the second property,  
{\em uniform monotonicity of cyclic subgroups}, is something that we establish only
for direct products of torsion-free hyperbolic groups. The third property,
having {\em{uniformly proper cyclics}},  is less stringent; it is of interest to us largely
because its relationship to the computability of the strengthened rel-cyclics
Dehn function $\delta^{o}(n)$ and the membership problem for cyclic sub-semigroups 
parallels the relationship between  
$\delta^c(n)$ and the membership problem for cyclic subgroups -- see Proposition \ref{p:semigps}.
   
\subsection{Uniformly quasigeodesic cyclics}\label{s:UQC}

In the following definition, the value of the constant $\lambda$ depends
on the choice of finite generating set for $G$ but its existence does not.

\begin{defn}\label{d:UQC}
We say that a finitely generated group $G$ has {\em uniformly quasigeodesic cyclics} (UQC)
if  there is a constant $\lambda>0$ such that $|g^n|_G\ge\lambda |n|$ for all
integers $n$ and all  $g\in G\ssm\{1\}$. 
\end{defn}

The function $n\mapsto |g^n|_G$ is sub-additive, so the limit of $|g^n|_G/n$ as $n\to\infty$
exists; it is called the {\em{algebraic translation number}} of $g\in G$,  
$$\tau(g) := \lim_n 
\frac 1 n |g^n|_G = \inf_{n\ge 1} \frac 1 n |g^n|_G.$$ 
This observation allows us to rephrase the (UQC) condition as follows.

\begin{lemma}\label{l:tau} A finitely generated group
$G$ has uniformly quasigeodesic cyclics (UQC) if and only if $\{\tau(g)\mid g\in G\ssm\{1\}\}$ is bounded away from $0$.
\end{lemma}

Torsion-free hyperbolic groups have uniformly quasigeodesic cyclics
-- see \cite{BH}, pp.~464--466. So too do
groups that act
freely, properly and cocompactly by isometries on {\rm{CAT}}$(0)$ spaces:
this follows from the \v{S}varc-Milnor Lemma and Proposition II.6.10(3) of \cite{BH}.
Other groups with UQC include the torsion-free subgroups
of automorphism groups of free groups \cite{alibegovic},  mapping class groups \cite{bow}, and 
hierarchically hyperbolic groups \cite{ahpz}.

\begin{lemma}\label{l:UQC} Let $G$ and $H$ be finitely generated groups and let 
$M<G$ be a finitely generated subgroup. 
If $G$ and $H$ have uniformly quasigeodesic cyclics, then so do $G\times H$ and $M$ and $G\dot\ast_M=(G,t \mid [t,m]=1\ \forall m\in M)$.
\end{lemma}

\begin{proof}
For suitable choices of word metrics, $|\mu|_G \le |\mu|_M$ and $|(g,h)|_{G\times H} = |g|_G + |h|_H$ for all $g\in G, h\in H$ and $\mu\in M$, so the only non-trivial assertion is
that $G\dot\ast_M$ has UQC if $G$ does.

We fix a finite generating set for $G$ and add the stable letter $t$ to 
get a generating set for $G\dot\ast_M$.  Note that the inclusion $G\hookrightarrow
G\dot\ast_M$ is an isometric embedding for the corresponding
word metrics.  

If a  non-trivial element $\gamma\in G\dot\ast_M$
is not  conjugate into $G$, 
then
it is conjugate to an element $\g_0$ represented by a word in cyclically reduced HNN form
$g_0 t^{e_0} \cdots g_lt^{e_l}$, with each exponent $e_i$   non-trivial and
$g_i\in G\ssm M$ or else $l=0$ and $g_0\in M$.
If $l=0$ then $(g_0t^{e_0})^n = g_0^nt^{ne_0}$. 
In the case $l>0$, for every integer $n>0$, the expression 
$(g_0 t^{e_0}\cdots g_{l}t^{e_l})^n$
is still in reduced HNN form. Regardless, by Britton's Lemma, any reduced word representing 
$\g_0^n$ will contain $n$ occurrences of $t^{\pm 1}$ for each occurrence in the reduced form of $\gamma_0$.
Thus $|\g_0^n|_{G\dot\ast_M}\ge  n$, whence $\tau(\g)=\tau(\g_0)\ge 1$.

If $\gamma$ is conjugate to $g\in G$, then $\tau(\gamma)=\tau(g)$
can be calculated in $G$ as well as $G\dot\ast_M$.   
Thus  
$$\inf\{\tau(\gamma) \mid \gamma\in G\dot\ast_M)\}=
\inf\{\tau(g) \mid g\in G\}$$
and $G\dot\ast_M$ has UQC if and only if $G$ does. 
\end{proof}

\begin{remark} The algebraic argument that we have just used to show that  $G\dot\ast_M$ has UQC 
can be rephrased in terms of the action of $G\dot\ast_M$ on the Bass-Serre tree of the splitting. A similar argument
shows that whenever  a finitely generated group $G$  is undistorted in an HNN extension $G\ast_M$,  the extension
will have UQC if and only if $G$ does.   
\end{remark}

The following simple observation will be crucial in the proof of our main results.

\begin{lemma}\label{l:uqc-n2}
If a finitely presented group $G$ has UQG, then $\d_G^{z}(n)= \d_G^{c}(n)= \d_G^{o}(n)\preceq \d_G(n^2)$.
\end{lemma}

\begin{proof} If $G$ has UQG then it is torsion-free and hence $\d_G^{z}(n)=\d_G^{c}(n)= \d_G^{o}(n)$.
In the torsion-free case,  the definition of $\d_G^{c}(n)$ involves the maximum of ${\rm{Area}}(wu^p)$ 
with $w=u^{-p}$, where $p\in\Z$ and $|w|+|u|\le n$.
If  $G$ has UQG, there  is a constant $\lambda>0$ such that $|u^p|_G\ge\lambda |p|$,
so $w=u^{-p}$ and $|w|\le n$  implies $|p|\le n/\lambda$. Therefore,  $|wu^p| \le n (1 + n/\lambda)$, 
so ${\rm{Area}}(wu^p)  \le \d_G(n (1 + n/\lambda))$ and 
$$\d_G^{c}(n) \le \d_G(n (1 + n/\lambda)) + n^2/\lambda \simeq \delta_G(n^2).$$
\end{proof}

\subsection{Uniformly monotone cyclics}

In the following definition, the value of the constant $k$ depends
on the choice of finite generating set for $G$ but its existence does not.

\begin{definition}\label{d:UMC}
A finitely generated group $G$ is said to have {\em{uniformly monotone
cyclics}} if
there exists a constant $k$ such that $|g^i|_G\le k|g^p|_G$
for all $g\in G$ and all integers $0<i<|p|$. 
\end{definition}

One might be tempted to think that this property would follow if one knew
that $G$ has uniformly quasigeodesic cyclics, but the following example
shows that it does not.

\begin{example}\label{ex:not-cat0}
Let $G=\< x,y\mid [x^2, y]\>$ be the group obtained from $\Z^2$ by freely
attaching a square root to a primitive element. It is
easy to see, for all $n\in\Z$, that $y^nxy^{-n}$ is a geodesic word in $G$.
But $(y^nxy^{-n})^2 = x^2$. Thus $G$ does not have uniformly monotone cyclics. 
On the other hand, $G$ does have UQG, because
the standard 2-complex of the given presentation can be endowed with a metric
of non-positive curvature by metrizing the 2-cell as a Euclidean rectangle
with sides labelled $x^2$ and $y$.  (Alternatively, one can observe that $G=\Z\dot\ast_{2\Z}$.)
\end{example}

We will need the following result in the proof of 
our main results. This observation does not appear to be in the literature.
The proof is less direct than one would like, but 
uses only standard techniques.

\begin{prop}\label{p:roots} Torsion-free hyperbolic groups have
uniformly monotone cyclics.   
\end{prop}

\begin{proof}  
Let $H$ be a torsion-free hyperbolic group and $h\in H$ a non-trivial element.
As $|h^j|_H = |h^{-j}|_H$, we restrict our attention to positive powers. 

If $H$ is free with a fixed basis and $u$ is the freely reduced word representing
$h$, then we pass to the cyclically reduced
form $w=x^{-1}u x$, where $x$ is the shortest prefix of  $u$ such that $w$ is cyclically reduced.
Geometrically, thinking of the Cayley tree for $H$, there is a geodesic line labelled $w^*$ passing through 
the vertex $x$ and $h$ is the vertex from which, following an arc labelled $x$, one meets the line at the vertex
$xw$.
Then, for all $p$, the unique reduced (geodesic) word representing $h^p$
is $x w^p x^{-1}$,  hence $|h^p| = 2 |x| + p|v|$, and in particular 
for each $0<i<p$ we have  $|h^i|< |h^p|$. 

\def\e{{\varepsilon}}

In the case of an arbitrary torsion-free hyperbolic group, we will be guided by 
the geometry behind the above proof (cf.~\cite{BH}, p.452), in particular we will consider analogues of $x$ and $w$
and will need to control the length of both. 

Let $h\in H\ssm\{1\}$. 
We consider those words $w$ in the generators of $H$ 
that are shortest  among all words representing elements
in the conjugacy class of $h$; suppose $w=x_w^{-1}h x_w$ in $H$.
Among these, we fix $w$ so that $|x_w|_H$ is minimal, 
and we fix a shortest word $x$ representing $x_w$. 
In contrast to the free case,  a non-trivial argument is needed to bound $|x|$ in terms of $|h^p|_H$
(cf.~example \ref{ex:not-cat0}).
As in the free case, there is a line in the Cayley graph labelled $w^*$ through the vertex $x_w$ and 
$x$ is the label on a shortest path from the vertex $1$ to this line.  (The line will not be 
a geodesic in general.) The following argument shows
that $x$ is also the label on a shortest path from the vertex $h^p$ to this line for all $p>0$. 

Let $x_1$ be any word that conjugates $h^p$ to an element
$u_p$ represented by a cyclic conjugate of $w^p$, where $p> 0$;
say $x_1^{-1} h^p x_1 = w_0^{-1}w^p w_0$, with $w_0$ a prefix of $w$.
Then, because roots are unique\footnote{this is a crucial use of the torsion-free hypothesis} in torsion-free hyperbolic groups, 
$x_1$ conjugates $h$ to $w_0^{-1}ww_0$, which being a cyclic conjugate of
$w$ has reduced length $|w|$. Therefore, 
\begin{equation}\label{e:least}
|x_1|\ge |x|=|x_w|_H,
\end{equation}
as $|x|$ was chosen to be minimal.

If $G$ is $\eta$-hyperbolic and  $|w|\ge 8\eta+1$, 
then the infinite paths in the Cayley graph labelled $w^*$ 
are $(8\eta+1)$-local geodesics
and hence $(l,\e)$-quasi-geodesics for suitable constants
$l,\e$ (see \cite{BH}, p.405). For each of the  
geodesics $w$ of length at most $8\eta$ representing elements of  $H\ssm\{1\}$,
the infinite paths labelled $w^*$ are quasi-geodesics whose constants
depend on $w$; since this set is finite, we can take absolute constants
covering them all.
Thus we may assume that $w^*$ is an 
$(l',\e')$-quasi-geodesic for all $w$ under consideration. We fix a constant $k_0>0$
such that every $(l',\e')$-quasi-geodesic triangle and quadrilateral in $H$
is $k_0$-thin (see \cite{BH}, p.402).

\smallskip

{\em{Claim}}: For all $p>0$,
either $|x| \le |h^p|_H + 2k_0+1$ or $|w^p|_H \le 4k_0+2$.
\smallskip

Fix a geodesic word $\sigma_p$ for $h^p$ and  consider the quadrilateral in the Cayley graph of $H$ with 
vertices $y_1, y_2, y_3, y_4$ and sides (read
in cyclic order) labelled $x^{-1},\sigma_p, x, w^{-p}$.

%%%FIGURE 2%%%

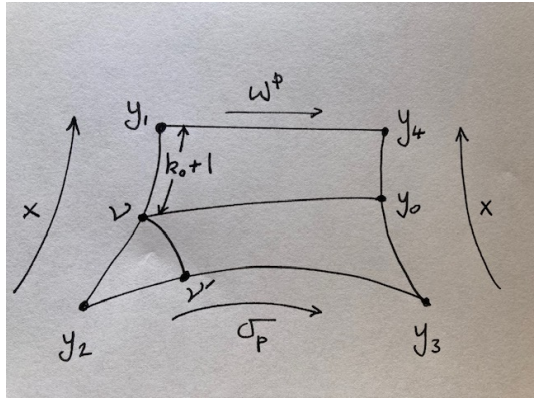
\begin{figure}[ht]
\begin{tikzpicture}[line cap=round, line join=round, >=Stealth]

% Rectangle corners
\coordinate (y1) at (0,3);
\coordinate (y2) at (0,0);
\coordinate (y3) at (5,0);
\coordinate (y4) at (5,3);

% Side points
\coordinate (z)  at (0,1.55);
\coordinate (y0) at (5,1.55);

% bottom point
\coordinate (n)  at (1.55,0);

% Top and bottom boundaries
\draw[thick] (y1) .. controls (1.7,3.28) and (3.3,3.28) .. (y4);
\draw[thick] (y2) -- (y3);

% Left and right sides of rectangle
\draw[thick] (y2) -- (y1);
\draw[thick] (y3) -- (y4);

% Middle curve from z to y0
\draw[thick] (z) .. controls (2.2,1.8) and (3.1,1.8) .. (y0);

% Points and labels
\fill (y1) circle (2pt) node[left] {$y_1$};
\fill (y2) circle (2pt) node[below] {$y_2$};
\fill (y3) circle (2pt) node[below] {$y_3$};
\fill (y4) circle (2pt) node[right] {$y_4$};
\fill (z)  circle (2pt) node[left] {$\nu$};
\fill (y0) circle (2pt) node[right] {$y_0$}; 
\fill (n) circle (2pt) node[below] {$\nu'$};

% Top labels 
\node at (2.6,3.6) {$w^p$};

% Top arrow (now below the labels)
\draw[-{Stealth[length=3mm]}] (1.2,3.35) -- (4.0,3.35);

% Bottom arrow and label
\draw[-{Stealth[length=3mm]}] (1.1,-0.7) -- (3.9,-0.7);
\node at (2.5,-1.03) {$\sigma_{p}$};

% Left and right x arrows, made straight
\draw[-{Stealth[length=3mm]}] (-1,0.4) -- (-1,2.7);
\node[left] at (-1,1.55) {$x$};

\draw[-{Stealth[length=3mm]}] (6,0.4) -- (6,2.7);
\node[right] at (6,1.55) {$x$};

% Circular arc from [y2,y3]   ending at z 
\draw[thick] (y2) ++(1.55,0) arc (0:90:1.55); % ends exactly at z

%\node at (-0.2,0.72) {$\ell$};

\end{tikzpicture}
        \caption{Bounding $|x|$}\label{fig2}
\end{figure} 

%%%% END FIGURE 2 %%%%%

An important point to observe is that there is no path in the Cayley graph that has length less than $|x|$ 
and connects either  $y_2$ or $y_3$ to the fourth side (the arc labelled $w^p$), for if there were then the label
$x_1$ on this path would contradict the minimality of $|x|$, as in the argument leading to (\ref{e:least}).

Suppose $|x|>k_0$ and consider the vertex $\nu$
on the first  
side that is a distance $k_0+1$ from $y_1$. By thinness of the 
quadrilateral, 
we know that $\nu$ is a distance at most $k_0$ from either the second or third side of the quadrilateral.
If it is the second side $[y_2,y_3]$, say $d(\nu,\nu')\le k_0$ with $\nu'\in [y_2,y_3]$, then 
$$|x| = d(y_1,y_2)\le d(y_1,\nu)+d(\nu, \nu') + d(\nu', y_2) 
\le k_0+1 + k_0 +  |h^p|$$ and we are done. If
it is the third side $[y_3,y_4]$, then let $y_0$ be a closest point to $\nu$.
Note that $d(y_3,y_0)\ge |x| - (2k_0 +1)$, for if not then 
$d(y_1,y_3) \le d(y_1, \nu) + d(\nu, y_0) + d(y_0, y_3) < |x|$,
contradicting the minimality of $|x|$, as in (\ref{e:least}).
It follows that $d(y_0, y_4)\le 2(k_0+1)$ and hence
$|w^p|_H\le d(y_1,\nu) + d(\nu, y_0) + d(y_0,y_4) \le 4k_0+2$. 
This finishes the proof of the Claim.

Let $\lambda$ be as in Definition \ref{d:UQC} and let $k_1=\lceil (4k_0+3)/\l\rceil$.
If we replace $p$ by $k_1p$ in the Claim, then the second case does not arise, so
for all $p>0$ and $h\neq 1$ we have
\begin{equation}\label{e:crude}
|x| \le   |h^{k_1p}|_H + 2k_0 +1 \le k_1\, |h^p|_H + 2k_0 + 1 \le 
(k_1 + 2k_0 + 1) |h^p|_H.
\end{equation}

Since the paths labelled $w^*$ are quasi-geodesics with uniform constants,
there exist $\mu_1,\mu_2>0$ such that $\mu_1 j |w| \le |w^j|_H \le \mu_2 j |w|$
for all $w$ and all $j\in\Z$, from which it follows that
$|w^i|_H \le (i\mu_2/p\mu_1)|w^p|_H$ for all $w$ and all $0<i<p$.

Thus, for all $h\neq 1$ and all $0<i\le p$, we have
$$
|h^i|_H = |xw^ix^{-1}|_H
\le 2|x| + |w^i|_H
\le 2|x| + \frac{i\mu_2}{p\mu_1} |w^p|_H
\le 2|x| + \frac{i\mu_2}{p\mu_1}(|h^p|_H + 2|x|),
$$ 
and the estimate (\ref{e:crude}) completes the proof.
\end{proof}

\begin{cor}\label{c:roots} If $H_1,\dots,H_n$
are torsion-free hyperbolic groups, then $H_1\times\dots\times H_n$
has  uniformly monotone cyclics.  
\end{cor}

\begin{proof}
If we take finite generating sets $A_i$ for $H_i$ and $A_1\cup\dots\cup A_n$
for $H=H_1\times\dots\times H_n$, then for all $h=(h_1,\dots,h_n)\in H$ and $i\neq 0$
we have $|h^i|_H = |h_1^i|_{H_1}+\dots + |h_n^i|_{H_n}$. Thus the corollary
follows immediately from the proposition.
\end{proof}

\subsection{Uniformly proper cyclics and cyclic semigroup membership}\label{s:UPC}

It is easy to verify that the following property is independent of the choice of finite generating set for $\G$, although
the exact function $\rho(n)$ is not.
 
\begin{defn}\label{d:depart}
Let $\G$ be a finitely generated group. We say that the cyclic subgroups of $\G$ 
are {\em uniformly proper} 
if there is a  recursive function $\rho:\N\to \N$ such that, for all 
$\g\in\G$ with $|\g|_\G\le n$,  either $\g$ is a torsion element of order less that $\rho(n)$, or
else $|\g^p|_G>n$ for all $p\ge\rho(n)$.
Equivalently, 
both the torsion-evolution function ${\mathfrak{t}}_\G(n)$ and
the cyclic return function ${\mathfrak{m}}_\G(n)$, as defined in (\ref{d:d^c}),  are bounded above by a recursive function of $n$. 
\end{defn}

\begin{prop}\label{p:semigps} For a finitely generated group $\G$, the
following conditions are equivalent:
\begin{enumerate}
\item[(i)] $\G$ has solvable word problem and uniformly proper cyclics;
\item[(ii)] in $\G$, there is uniform algorithm
to solve the membership problem for the cyclic sub-semigroups.
\end{enumerate} 
If $\G$ is finitely presented, then each of these conditions is equivalent to
\begin{enumerate}
\item[(iii)] $\d_\G^{o}(n)$ is a recursive function.
\end{enumerate}
\end{prop}

\begin{proof} First we prove $(ii)\implies (i)$.
Suppose that there is an algorithm that, given words $x$ and $y$ in the generators, will decide whether or not
there is an integer $p>0$ such that $x=y^p$ in $\G$.
Taking $y$ to be the empty word solves the word problem in $\G$. Taking $y=x^{-1}$
enables us to determine if $x$ has finite order in $\G$ or not. If $x$ does have finite
order, then we can use the solution to the word problem to determine its order;
by definition,  ${\mathfrak{t}}_\G(n)$ is the maximum of the orders among words of length $n$.
If $x$ has infinite order in $\G$, then for each word $w$ with $|w|\le n$ the algorithm
identifies whether or not $w$ equals a positive power of $x$; if it does,
the solution to the word problem finds this unique power $p_w(x)$.  
Let $\rho_0(n)$ be the maximum of these $p_w(x)$ with $|x|\le n$ and $|w|\le n$. Then $\rho(n)=\rho_0(n)+{\mathfrak{t}}_\G(n)+1$
is a recursive function with the properties required in Definition \ref{d:depart}.

Conversely, if $\G$ has a solvable word problem and $\rho(n)$ is as in definition
\ref{d:depart}, then given words  $x,y$ with $\max\{|x|, |y|\}\le n$, we use the solution to 
the word problem to check whether or not $x=y^p$ in $\G$ for  $p=1,\dots,\rho(n)$,
and conclude that $x$ is not in the semigroup generated by $y$ if none of these
equalities holds.

We prove $(iii)\implies (i)$ if $\G$ is finitely presented.  Assume $\d_\G^{o}(n)$ is recursive.
By definition, $\d_\G(n)\le \d_\G^{o}(n)$,
and a finitely presented group has a solvable word problem if and only if
its Dehn function is (equivalently, is bounded above by) a recursive function. Thus 
$\G$ has a solvable word problem. The presence of the summand $|pn|$
in the definition of $\d_\G^{o}(n)$ ensures that if $\g$ is a torsion
element and $|\g|_\G\le n$ then the order of $\g$ is less than $\d_\G^{o}(n)$, and 
so   ${\mathfrak{t}}_\G(n)$ is bounded above by this recursive
function (Remark \ref{r:2.3}).  It follows that ${\mathfrak{t}}_\G(n)$ can be calculated exactly by using the
solution to the word problem to check whether $u^p=1$
for each word of length at most $n$ and each positive integer $p\le \d_\G^{o}(n)$.
The inclusion of the $|pn|$ term in $\d_\G^{o}(n)$ also ensures that 
$\dc_\G (n) \le \d_\G^{o}(n)$, from which it follows that $\dc_\G (n)$ is also recursive.

Finally, we prove $(i)\implies (iii)$ if $\G$ is finitely presented.  Given $n\in\N$, we list pairs of 
words $u,w$ with $\max\{|u|, |w|\}\le n$ and for each positive integer $p\le {\mathfrak{t}}_\G(n)$
we use the solution to the word problem to determine if $u^p=1$ in $\G$.
Thus we determine the order of each $u$. If $u$ has order $o$
then for each positive $p< o$ we use the solution to the problem to decide if
$w=u^p$. If it does, then we calculate ${\rm{Area}}(wu^{-p})$:
this can be done by naively testing all free equalities of the form (\ref{w:equ}) until one
is found to be valid, then searching for an
equality of the same form with fewer factors,  testing only the finitely many
that satisfy the conclusion of the  Bounded Noise Lemma \ref{l:BNL}.

If $u$ is found to have infinite order, then for each positive $p\le \dc_\G (n)$
we use the solution to the word problem to determine whether or not $w=u^p$ or 
$w=u^{-p}$ in $\G$. Whenever a valid equality is found, we proceed as in 
the previous paragraph to calculate ${\rm{Area}}(wu^{\pm p})$.
When these calculations are done, we have computed a finite set of areas,
each with an associated integer $p$. 
By taking the maximum of these areas, after adding $|pn|$ to each, we have computed
$\d^{o}(n)$.
\end{proof}

\section{Lower bounds on absolute conjugator length}\label{s:dist-CL}

We once again work with an exact sequence $1\to N\to G\to Q\to 1$ and its
fibre product $P<G\times G$, with $G$ torsion-free and hyperbolic and $Q$ 
finitely presented.  We fix   $A\subseteq N$ and presentations $G = \<X\mid R\>$
and $Q=\<X\mid R\cup A\>$ with $A\subset X$,  as in Section \ref{s:3}, and we work with the generating set
for $P$ that is described in Lemma \ref{l:3.1}, as well as the obvious generators for $G\times G$.
Note that $|(g_1, g_2)|_P \ge \frac{1}{2}|(g_1, g_2)|_{G\times G}$ for all $(g_1,g_2)\in P$. 

The following lower bound does not require any control on the geometry of cyclic
subgroups in $Q$.

\begin{prop}\label{p:naive-lower}
$\dis (n)\preceq \CL_P(n)$.
\end{prop}

\begin{proof} 
We fix a non-trivial element $a$ of our normal generating set $A$ of $N$.
Because $G$ is torsion-free and hyperbolic, the centraliser of $a$
is infinite cyclic, generated by $a_0$ say, with $a_0^e=a$ and $e>0$. 
Let $\e>0$ be the maximum value of $|(a_0^r,a_0^s)|_P$ for
those elements $(a_0^r,a_0^s)$ with $-e< r,s<e$ that lie in $P$.
For all $p,q$ such that $(a_0^p,a_0^q)\in P$, we  
write $p= p'e + r$ and $q= q'e + s$ with $-e< r,s < e$. %,\ |p'|\le |p|, \ |q'|\le |q|$, 
From the triangle inequality we get:
\begin{equation}\label{lb0}
|(a_0^p,a_0^q)|_P \le  |(a^{p'}, a^{q'})|_P +  |(a_0^r,a_0^s)|_P \le 
|(a^{p'},1)|_P + |(1,a^{q'})|_P + \e \le |p'|+|q'|+\e \le |p|+|q| +\e.
\end{equation}

Scholium \ref{l:Ndoes} provides us, for each $n>0$,
with an element
$\gamma\in N$ such that $|\gamma|_G\le n$ and $|(\g,1)|_P\ge \dis(n) -n$.
We consider 
the elements $\alpha=(a,a)$ and $(\g^{-1}a\g, a)$. Note that $|(a,a)|_P=1$
while $|(\g^{-1}a\g, a)|_P\le 2n +3$ because $|(1,a)|_P \le |(a,a)|_P + |(a,1)|_P=2$ and
$$(\g^{-1}a\g, a)= (\g,\g)^{-1}(a,1)(\g,\g) (1,a).$$

The centralizer $C(\alpha)$
of $\alpha:=(a,a)$ in $G\times G$ is  $\<(a_0,1), (1,a_0)\>$,
so  the elements conjugating
$\alpha$ to $(\g^{-1}a\g, a)$ in $G\times G$ are $(a_0^p\gamma , a_0^q)$
with $p,q\in\Z$. Since $(\g,1)\in P$, such an element will lie
in $P$ if and only if $(a_0^p,a_0^q)\in P$. 

Because $G$ has uniformly quasigeodesic cyclics, there is a constant
$\lambda\in (0,1)$ such that $|a_0^m|_G\ge \lambda |m|$ for all integers $m$.

If $|p|+|q|\ge \half \dis (n)$ then
\begin{equation*}\label{low1}
|(a_0^p\gamma , a_0^q)|_{G\times G} = | a_0^p\g|_G + |a_0^q|_G \ge
 |a_0^p|_G - |\g|_G +  |a_0^q|_G \ge \lambda (|p| + |q|) - n
 \ge \frac{\l}{2} \dis (n) - n,
\end{equation*}
which implies that $|(a_0^p\g, a_0^q)|_P \ge \frac{\lambda}{4} \dis (n) - \frac{1}{2}n $ if 
$(a_0^p, a_0^q)\in P$.

If  $|p|+|q|\le \half \dis (n)$ and $(a_0^p, a_0^q)\in P$, then
\begin{equation*}\label{low2}
|(a_0^p\gamma,  a_0^q)|_P \ge |(\g,1)|_P - |(a_0^p,a_0^q)|_P
\ge  |(\g,1)|_P  - (|p|+|q|+\e)
\ge \dis (n) -n  - {\frac{1}{2}} \dis (n) -\e,
\end{equation*}
where the second inequality comes from (\ref{lb0}) and the third from our choice of $\gamma$.

Thus,  every element of $P$ that conjugates $(a,a)$ to $(\g^{-1}a\g, a)$
must have length (in $P$)  at least ${\frac{\lambda}4} \dis (n) - n -\e$.
\end{proof}

\section{Lower bounds on relative conjugator length}\label{s:rel-lower}
The results in this section are needed for the sentence in
Theorem \ref{t:main} concerning the case where $Q$ is torsion-free.

The conjugator length function $\CL_H(n)$ is an instrinsic invariant of $H$,
but when $H$ is a subgroup of $\G$ one might also be interested in the following 
extrinsic  invariant. Note that here, the length of the elements being conjugated is measured
in $\G$ while the length of the conjugator is measured in $H$.

\begin{defn} Let $H$ be a finitely generated subgroup of the finitely generated group $\G$.
$$
\CL_H^\G(n) := \max_{u,v} \min_\g \{ |\g|_H \colon u,v,\g\in H,\ \g^{-1}u\g=v,\ 
|u|_\G+|v|_\G\le n\}.
$$
\end{defn}
It is clear from the definitions that
$\CL_H(n) \preceq \CL_H^\G(n)$, because the maximum in the definition of the latter 
is taken over a larger set (when suitable generators are chosen).

Once again, in the following statement
$P<G\times G$ is the fibre product associated to a short exact sequence $1\to N\to G\to Q\to 1$, with
generators for $P$ and presentations for $G$ and $Q$  chosen as  
 in Section \ref{s:3}. The functions $\d_Q^z(n)$ and $\m_Q (n)$ were defined
in (\ref{d:d^c}).

\begin{prop}\label{p:relLower}
If $Q$ is finitely presented and $G$ is torsion-free and
hyperbolic, then $\d_Q^z(n)\preceq \CL_P^{G\times G}(n) + n\, \m_Q(n)$
and $\m_Q(n) \preceq \CL_P^{G\times G}(n)$.
\end{prop}

\begin{proof} We have a finite generating
set $X$ for $G$. We consider
all words $\tilde\g ,\tilde{u}$ of length at most $n$ in the free group $F(X)$ for which 
there is an integer $m\ge 0$ such that $\tilde{u}^m\tilde\g=1$ in $Q$ and $\tilde{u}$ has
infinite order in $Q$.  
We will prove the proposition by identifying constants $c_0,\, c_1$ and $c_2$ so that  
\begin{equation}\label{e:target}
{\rm{Area}}_Q (\tilde{u}^{m}\tilde{\g} ) \le c_0 \ \CL_P^{G\times G}(6n) + c_1\, (m+1)n\ \ \ \text{  and  }\ \ \ 
m \le c_2 \, (\CL_P^{G\times G}(6n) + n).
\end{equation} 

Let $\g, u\in G$ be the images of
$\tilde\g ,\tilde{u}$.
Note that $({u}^m{\g},1)\in N\times\{1\}<P$.  
Consider $U=(u,u)$ and $V=({\g}, 1)^{-1} (u,u) ({\g}, 1)$
and note that $|U|_{G\times G}\le 2n$ while $|V|_{G\times G}\le 4n$.
Also, $V$ is conjugate to $U$ in $P$, because  
$$
V=({\g}, 1)^{-1} (u,u) ({\g}, 1) = ({u^m\g}, 1)^{-1} U ({u^m\g}, 1).
$$  
As $u\in G$ is non-trivial, its centraliser $C_G(u)$ is cyclic, generated
by $u_0$ say.  

The elements of $G\times G$ that conjugate $U$ to $V$
are  $\zeta_{pq} = (u_0^p, u_0^q) ({u}^m{\g},1)$ with $p,q\in\Z$.
The image of $u_0$ in $Q$ has infinite order, so $(u_0^p,u_0^q)\in P$
if and only if $p=q$. Hence $\zeta_{pq}\in P$ if and only if $p=q$. 
And by definition,  since $|U|_{G\times G} + |V|_{G\times G}\le 6n$, for some $q$ we have 
\begin{equation}\label{e:zq}
\ell_q:=|\zeta_{qq}|_P \le \CL_P^{G\times G}(6n).
\end{equation}

Let $W$ be a geodesic word for $\zeta_{qq}$ in the generators $(x,x), (a,1)$
of $P$ and let $W_q$ be the word obtained from $W$ by deleting 
all occurrences of the generators $(a,1)$. Reading the second coordinate we see
that $(u_0^q, u_0^q)= W_q$ in $G\times G$. Reading the
first coordinate then gives us ${u}^m{\g} = w_q^{-1} w$ in $G$, 
where $w$ and $w_q$ are the words in $F(X)$ obtained from $W$ and $W_q$ by replacing $(x,x)$ with $x$
for all $x\in X$ and $(a,1)$ with $a$ for all $a\in A$.  Note that $|w_q^{-1} w|\le 2\ell_q$. 
We now argue as in  the proof of Theorem \ref{t:Pdist}: the linear isoperimetric inequality
in $G$ provides a constant $C$ such that $\tilde{u}^m{\tilde{\g}} (w_q^{-1}w)^{-1}$
is a product in the free group $F(X)$ of at most $C (nm+n + 2\ell_q)$ conjugates of relators
for $G$; and by construction,  $w_q^{-1}w$
is freely equal to a product of $\ell'$ conjugates of the generators $a\in A^{\pm 1}$
(which are relators in $Q$),
where $\ell'\le \ell_q$ is the number of occurrences of these generators in $w$;
therefore, 
$$\Area_Q (\tilde{u}^{m}\tilde{\g} ) \le C (nm+n + 2\ell_q) +\ell_q.$$
We chose $q$ so that (\ref{e:zq}) holds,  so to satisfy (\ref{e:target}) it suffices to define $c_0:= (2C+1)$ and $c_1:=C$.

A further argument is required to identify $c_2$.
Let $e\ge 1$ be such that $u_0^e=u$ and observe that for all $q$ we have
\begin{align*}
2|\zeta_{qq}|_P = 2|(u_0^{me+q}\gamma, u_0^q)|_P
& \ge |(u_0^{me+q}\gamma, u_0^q)|_{G\times G}\\
& =  |u_0^{me+q}\gamma|_{G}+ |u_0^q|_{G} \\
&\ge |u_0^{me+q}|_G + |u_0^q|_G - |\g|_G \\
&\ge \lambda (|me+q| + |q|) -n,
\end{align*}
where $\lambda>0$ is the UQC constant for $G$ (Definition \ref{d:UQC}). 
We take $q$ so that (\ref{e:zq}) holds,  in which case
$$
|me+q| + |q| \le \frac{1}{\l}(2\, \CL_P^{G\times G}(6n) + n).
$$
Thus, since $e\ge 1$, 
$$
|m| \le |me| \le \frac{1}{\l} (2\, \CL_P^{G\times G}(6n) + n),
$$
which gives us $c_2$.
\end{proof}

All of the functions appearing in Proposition \ref{p:relLower} grow at least linearly, and for functions that grow
at least linearly, $f(n)\preceq g(n)$ implies $n\, f(n)\preceq n\, g(n)$.

\begin{cor}\label{c:relLower}
$\d_Q^z(n)\preceq n\ \CL_P^{G\times G}(n)$.
\end{cor}

\section{Upper bounds on conjugator length}\label{s:upper}

We continue with the convention that $P<G\times G$ is the fibre product of an epimorphism $G\twoheadrightarrow Q$
and that generators and relators for these groups have been chosen as in Section \ref{s:3}.

\begin{prop}\label{p:upper}
If $G$ is torsion-free and hyperbolic, then 
$$\CL_P^{G\times G}(n) \preceq \d_Q^{c}(n).$$
\end{prop}

\begin{proof}
Suppose that $U,V\in P$ are conjugate in $P$ and that $|U|_{G\times G}+ 
|V|_{G\times G}\le n$.
If $U\in G\times 1$ then $V\in G\times 1$ and $g^{-1} U g =V$
for some $g=(\gamma,1)\in G\times 1$ with 
$|\gamma|_G \le \CL_{G}(n)\simeq n$. 
In this case, $(\gamma,\gamma)$ also conjugates $U$ to $V$, and 
$|(\gamma,\gamma)|_P = \frac{1}{2}|\gamma|_G$.

A similar argument applies if $U,V\in 1\times G$, so we may assume that
$U=(u_1,u_2)$ and $V=(v_1,v_2)$ with $u_1,v_1,u_2,v_2\in G$  
all non-trivial and $|u_1|_G+|u_2|_G+|v_1|_G+|v_2|_G\le n$.

We reduce to the case $u_2=v_2$ by arguing as follows. 
Fix $C>0$ so that $\CL_G(n)\le Cn$.
As $U$ and $V$ are conjugate in $G\times G$, there exists $g\in G$
such that $g^{-1}u_2g = v_2$ with $|g|_G\le Cn$. 
Then $(g,g)^{-1}\in P$ conjugates $V$ to $V'=(gv_1g^{-1}, u_2)$
and $|V'|_{G\times G} \le n(2C +1)$. If we can conjugate $U$
to $V'$ by $z\in P$ then we can conjugate $U$ to $V$ by $z(g,g)$,
and this is efficient enough because $|z(g,g)|_P \le |z|_P + |(g,g)|_P \le |z|_P+Cn$.

With this reduction in hand, and to avoid being overwhelmed by constants,
 we assume $U=(u_1, u_2)$ and $V=(v_1, u_2)$ with 
$|U|_{G\times G} + |V|_{G\times G}\le n$. 
We fix $\gamma\in G$ with $|\gamma|_G\le Cn$ so that $\g^{-1}u_1\gamma = v_1$,
hence $(\g,1)^{-1} U (\g,1) = V$.

The centralisers of $u_1,u_2\in G$ are cyclic, generated by $y_1$ and $y_2$ say,
so the set of elements of $G\times G$ that conjugate $U$ to $V$ is
$$
\mathcal{Z}_{UV}=\{ (y_1^{q}\gamma, y_2^{q'}) \colon q, q'\in \Z\}.
$$
By hypothesis, at least one of these elements lies in $P$. We fix 
$\zeta=(y_1^{q_1}, y_2^{q_2})(\gamma,1)\in P\cap \mathcal{Z}_{UV}$. Our goal
is to manipulate $\zeta$ so as to produce $\zeta''\in P\cap \mathcal{Z}_{UV}$
with a suitable upper bound on its length.

For some $e_i>0$ and $i=1,2$ we have $y_i^{e_i}=u_i$.
Write $q_2 = p e_2 + r_2$  
with $0\le r_2 < e_2$.
Then, observing that the element on the left lies in $P\cap \mathcal{Z}_{UV}$,
consider the following equation, where $p':=q_1 -p e_1$:
\begin{equation}\label{u1}
(u_1, u_2)^{-p}\zeta  = 
(y_1^{-pe_1}, y_2^{-pe_2})\ (y_1^{q_1}, y_2^{q_2})(\gamma,1) =
(y_1^{p'}, y_2^{r_2})(\gamma,1).
\end{equation}
Let $\zeta':= (y_1^{p'}, 1)\ (1,y_2^{r_2})(\gamma,1)\in P\cap \mathcal{Z}_{UV}$.

If the image of $y_1$ in $Q$ has infinite order, or finite order $\omega\ge 2|p'|$, then define
$\zeta''=\zeta'$ and $p''=p'$. Otherwise, we write $p' = \mu \omega + p''$ with $\mu\in\Z$ and
$|p''|\le \omega/2$ and define $\zeta'' := (y_1^{\omega}, 1)^{-\mu}\zeta'$, noting that
$(y_1^{\omega}, 1)\in P$ and hence $\zeta''\in P\cap \mathcal{Z}_{UV}$. 
$$
\zeta'' = (y_1^{p''}, 1)(1,y_2^{r_2})(\gamma,1).
$$

We want to bound $|\zeta''|_P$ in terms of $\delta_Q^{c}(n)$. 
We chose $\g$ with  $|\g|_G \le Cn$.  Proposition \ref{p:roots} provides
a constant $k$ such that $|y_2^{r_2}|_G\le k |u_2|_G
$ and $|y_1|_G\le k |u_1|_G$, so $|y_2^{r_2}|_G + |y_1|_G \le kn$.
We now choose geodesic words $\sigma$ and $\tau$ in the generators $X$ of $G$ that
represent 
$y_2^{r_2}\g^{-1}$ and $y_1$  respectively.  
\begin{equation}\label{es0}
|\sigma| \le |y_2^{r_2}|_G + |\g|_G \le (C+k)n \text{   and  } |\tau|\le kn.
\end{equation} 

As $\zeta''\in P$, we have $\sigma = \tau^{p''}$ in $Q$, and we have arranged that 
$|p''| $ is at most half the order of $\tau$ in $Q$. As
$|\sigma| + |\tau| \le (C+2k)n$, from the definition of $\delta_Q^{c}(n)$ we have
\begin{equation}\label{es1}
\Area_Q (\sigma \tau^{-p''}) + |p''| (C+2k)n \le \d_Q^c((C+2k)n).
\end{equation}
From Scholium \ref{scholiumNew} we have the estimate
\begin{equation}\label{es2}
|\zeta''|_P \le (L+1)\ \Area_Q(\sigma \tau^{-p''}) +  2|\sigma \tau^{-p''}|,
\end{equation}
where $L$ is the longest relator in the presentation of $Q$.  And combining  (\ref{es0}) with the triangle 
inequality we have
\begin{equation}\label{es3}
|\sigma \tau^{-p''}| \le |\sigma| + |p''|.|\tau| \le |p''| (C+2k)n.
\end{equation}
The estimates (\ref{es1}),  (\ref{es2}), and (\ref{es3}),  tell us that we have found
$\zeta''\in P$ that conjugates $U$ to $V$ with
$$|\zeta''|_P\le (L+1)\,\d_Q^{c}((C+2k)n),$$
and this concludes the proof.
\end{proof}

\section{Theorem \ref{t:main} and Corollary \ref{c:forBR}}\label{s:summary}
 
\mainthm*

\begin{proof}
The inequality $\d_Q(n)\preceq \CL_P(n)$ follows immediately from Theorem \ref{t:Pdist}, which was proved in Section \ref{s:3},  and Proposition \ref{p:naive-lower}. The inequality $\CL_P(n)\preceq \CL_P^{G\times G}(n)$ is immediate from the definitions,
and the inequality $\CL_P^{G\times G}(n) \preceq \d_Q^{c}(n)$ was established in Proposition \ref{p:upper}. 

If $Q$ is torsion-free, then $\d_Q^c(n)=\d_Q^z(n)$,  so Corollary \ref{c:relLower} tells us that 
$\d_Q^c(n)\preceq n\,  \CL_P^{G\times G}(n)$. 

If $Q$ has uniformly quasigeodesic cyclics, then $\mathfrak{m}_Q(n)\simeq n$, so Proposition \ref{p:relLower} 
allows us to sharpen the previous estimate to $\d_Q^c(n)\preceq  \CL_P^{G\times G}(n)$, provided 
$\CL_P^{G\times G}(n)\succeq n^2$,  and this complements our earlier bound $\CL_P^{G\times G}(n) \simeq \d_Q^{c}(n)$.
\end{proof}

\maincor*

\begin{proof}
Recall that  $Q$ is said to have uniformly quasigeodesic cyclics if there is a constant $\lambda>0$ such
that $|u^p|_G\ge \lambda |p|$ for all non-trivial $u\in Q$.  This means that the 
absolute value of the exponents $p$  considered in the
calculation of $\d_Q^c(n)$ (Definition \ref{d:d^c}) can be at most $n/\lambda$, hence the words $wu^p$
whose areas contribute to $\d_Q^c(n)$ each have length at most $n + n^2/\lambda$, 
and the additional summand $|pn|$ is at most $n^2/\lambda$. Therefore $\d_Q^c(n)\preceq \d_Q(n^2)$ and the
corollary follows from the bounds in the theorem.
\end{proof}

\begin{remark}\label{r:last}
 In the preceding proof,  we estimated ${\rm{Area}}(wu^p)$ by $\d_Q(n^2)$, roughly speaking. 
When one has more control over the geometry of $Q$, this can quickly be improved.  
In particular,  if $Q$  
has uniformly monotone cyclics (Definition \ref{d:UMC}),   then there is a constant $k$ such that
$|u^i|_G\le k|u^p|_G=k|w|_G\le kn$ for all  $0<i<|p|$,  so  the loop in the
Cayley graph of $Q$ labelled $wu^p$ never leaves the ball of radius $(k+1)n$ about the identity. 
In many groups, this allows one to bound ${\rm{Area}}(wu^p)$ by something sharper than $\delta_Q(n^2)$.
For example,  if $Q=\Z^2$ then $\delta_Q(n^2)\simeq n^4$ but ${\rm{Area}}(wu^p)\preceq n^3$
if $wu^p$ is a word of length $\preceq n^2$ that labels a path in the ball of radius $\preceq n$.   
The potential of such refinements will be explored in \cite{BR} and its sequel.
\end{remark}

\section{Encoding Dehn functions into conjugator length functions}\label{s:last}

The {\em Rips Construction} \cite{rips}
is an algorithm that associates to any finite group-presentation 
$\mathcal{Q}\equiv \< X\mid {R}\>$ a short exact sequence
$$
1\to N\to G \overset{p}\to Q\to 1
$$
where $Q$ is the group presented by $\mathcal{Q}$ while
$N$ is a 2-generator group and $G$ is a torsion-free hyperbolic group
that satisfies a prescribed small-cancellation condition.
The algorithm will output a
finite aspherical presentation $\< X,a,b\mid T\>$ for $G$ such that $N=\<a,b\>$
and $|T|= 4|X| + |R|$, where the identity map on $X$ induces 
$p:G \to Q$. It is easy to check that the fibre product $P<G\times G$ associated to $p$
is generated by $(a,1),\, (b,1)$ and $\{(x,x) \mid x\in X\}$.

The 1-2-3 Theorem \cite{BBMS}  tells us that  if $Q$
is of type ${\rm{F}}_3$ then the fibre product $P<G\times G$ of $G\twoheadrightarrow Q$ 
is finitely presented. 

In order to prove Theorem \ref{t:final}, we will apply the Rips construction to a group
$Q$ of type ${\rm{F}}_3$ with uniformly quasigeodesic cyclics and controlled Dehn function.

\subsection{Forcing $Q$ to have uniformly quasigeodesic cyclics}

We restate Proposition \ref{p:makeUQC} for the reader's convenience.

\makeUQC*

In the course of the proof we shall need the following connection between
distortion and Dehn functions. Here, the notation $\G\dot\ast_H$
denotes the trivial HNN extension over $H$, that is
$$
(\G, t \mid t^{-1}ht=h \ \forall h\in H).
$$

\begin{theorem}[\cite{BH}, III.$\G$.6.20]\label{hnn}
If $\G$ is a finitely presented group and $H<\G$ is a finitely
generated subgroup, then  
$$
\max \{ \d_\G(n), \, n\ \dist_H^\G(n)\} \preceq \delta_{\G\dot\ast_H}(n)
\preceq n\ \delta_\G\circ\dist_H^\G(n).
$$
\end{theorem}

\noindent{\bf{Proof of Proposition \ref{p:makeUQC}.}}
Given a finite presentation for $Q$, we perform the Rips Construction 
to get $1\to N\to G\to Q\to 1$ with
a finite aspherical presentation for $G$ and hence a finite presentation 
$\< \tilde{X} \mid \tilde{T}\>$ for $G\times G$.
We also get an explicit finite generating set $S$ for the fibre
product $P<G\times G$
and can therefore present the HNN extension $Q^\dagger:=(G\times G)\dot\ast_P$
as
$$ Q^\dagger = \< \tilde{X}, t \mid \tilde{T}, \ [t^{-1},s]\ (s\in S)\>.$$
In Therorem  \ref{t:Pdist} we calculated $\dist_P^{G\times G} (n)\simeq \d_Q(n)$, so  
$$n\, \d_Q(n) \preceq \d_{Q^\dagger} (n) \preceq n\, (\d_Q(n))^2,$$
by Theorem \ref{hnn}, since $\d_{G\times G}(n)\simeq n^2$.  

$N$ is finitely generated, so if $Q$ is of type ${\rm{F}}_3$, then $P$
is finitely presented by the 1-2-3 Theorem \cite{BBMS}. Moreover $G$ has an aspherical
presentation, so $G\times G$ is of type ${\rm{F}}_m$ for all $m\in\N$.
A standard topological argument
shows that if $\G$ is type ${\rm{F}}_3$ and
$H<\G$ is finitely presented (i.e.~type ${\rm{F}}_2$), then
$\G\dot\ast_H$ is type ${\rm{F}}_3$. So if 
$Q$ is of type ${\rm{F}}_3$, then so is $ Q^\dagger$. Killing the stable letter $t$ retracts $Q^\dagger$ onto
$G\times G$,  so we can map it onto $G$ and hence $Q$.

Finally, since $G$ is torsion-free and hyperbolic, it has uniformly
quasigeodesic cyclics. And by Lemma \ref{l:UQC}, this 
property is inherited by $Q^\dagger$.
\qed

\subsection{Involving all standard Dehn functions}

\final*

\begin{proof}
Let $Q$ be a group
of type ${\rm{F}}_3$ with $\d_Q(n) = \d(n)$. Using Proposition \ref{p:makeUQC},
we exchange $Q$ for a group $Q^\dagger$ of type ${\rm{F}}_3$
that has uniformly quasigeodesic cyclics and 
\begin{equation}\label{e:n}
n \d(n) \preceq \d_{Q^\dagger} (n) \preceq n (\d(n))^2.
\end{equation}
By applying the Rips construction to $Q^\dagger$, we obtain
a hyperbolic group $H$ and an epimorphism
$H\onto Q^\dagger$ with finitely generated kernel. According to the
1-2-3 Theorem, the associated fibre product $P_\dagger<H\times H$ is finitely presented, and
according to Corollary \ref{c:forBR},
\begin{equation}\label{e:nn}
\d_{Q^\dagger} (n) \preceq \CL_{P_\dagger}(n) \preceq \d_{Q^\dagger}(n^2).
\end{equation}
Together, (\ref{e:n}) and  (\ref{e:nn}) finish the proof.
\end{proof}

 \end{document}